\documentclass[12pt, a4paper]{article}
\usepackage{}

\usepackage{amsfonts}
\usepackage{bbm}
\usepackage{mathrsfs}
\usepackage{latexsym}
\usepackage{graphicx}
\usepackage{amssymb}
\usepackage{amsmath}
\usepackage{color,xcolor}
\usepackage{mathtools}

\newtheorem{theorem}{Theorem}[section]
\newtheorem{lemma}[theorem]{Lemma}
\newtheorem{corollary}[theorem]{Corollary}
\newtheorem{proposition}[theorem]{Proposition}

\newtheorem{problem}[theorem]{Problem}

\def\det{{\rm det}}

\title{{\Large \bf On the spectral radius, energy and Estrada index of the Sombor matrix of graphs \thanks{Supported by the National Natural Science Foundation of China (No. 12071411, 11771443).}~}}
\author{Zhen Lin\thanks{Corresponding author. E-mail addresses: lnlinzhen@163.com (Z. Lin), miaolianying@cumt.edu.cn (L. Miao).}, Lianying Miao\\
{\footnotesize School of Mathematics, China University of Mining and Technology,}\\ {\footnotesize  Xuzhou, 221116, Jiangsu, P.R.
China}\\
}

\date{}
\begin{document}
\openup 1.0\jot
\date{}\maketitle
\begin{abstract}

Let $G$ be a simple undirected graph with vertex set $V(G)=\{v_1, v_2, \ldots, v_n\}$ and edge set $E(G)$.
The Sombor matrix $\mathcal{S}(G)$ of a graph $G$ is defined so that its $(i,j)$-entry is equal to $\sqrt{d_i^2+d_j^2}$ if
the vertices $v_i$ and $v_j$ are adjacent, and zero otherwise, where $d_i$ denotes the degree of vertex $v_i$ in $G$. In this paper, lower and upper bounds on the spectral radius, energy and Estrada index of the Sombor matrix of graphs are obtained, and the respective extremal graphs are characterized.


Sombor Estrada index

\end{abstract}
\baselineskip 20pt

\section{\large Introduction}
\ \ \ \  Let $G$ be a simple undirected graph with vertex set $V(G)$ and edge set $E(G)$. For $v\in V(G)$, $N_G(v)$ denotes the neighborhood of $v$ in $G$, and $d_v=|N_G(v)|$ denotes the degree of vertex $v$ in $G$. The minimum and maximum degree of a vertex in $G$ are denoted by $\delta$ and $\Delta$, respectively.

Given a graph $G$, the adjacency matrix $A$ is the $n\times n$ matrix whose $(i, j)$-entry is $1$ if $v_iv_j\in E(G)$ and zero otherwise. Thus $A$ is a real symmetric matrix, its eigenvalues must be real and arranged in non-increasing order $\lambda_1\geq \lambda_2\geq \cdots \geq \lambda_n$, where $\lambda_1$ is called the spectral radius of $G$. For the adjacency spectra, one may refer to \cite{N, S, TT} and the references therein.

The energy of $G$, introduced by Gutman \cite{G}, is defined as $E_{A}(G)=\sum\limits_{i=1}^{n}|\lambda_i|$, which is intensively studied in chemistry, since it can be used to approximate the total $\pi$-electron energy of a molecule, see for example \cite{G1, G2}. There is a wealth of literature relating to the energy see, for example, \cite{G3} for surveys, and see \cite{LSG}, for monograph.

In 2021, a new vertex-degree-based molecular structure descriptor was put forward by Gutman \cite{G4}, the Sombor index of a graph $G$, defined as $SO(G)=\sum_{uv\in E(G)}\sqrt{d_u^2+d_v^2}$. The study of the Sombor index of graphs has quickly received much attention. Cruz et al. \cite{CGR} studied the Sombor index of chemical graphs, and characterized the graphs extremal with respect to the Sombor index over the following sets: (connected) chemical graphs, chemical trees, and hexagonal systems. Deng et al. \cite{DTW} obtained a sharp upper bound for the Sombor index among all molecular trees with fixed numbers of vertices, and characterized those molecular trees achieving the extremal value. Das et al. \cite{DCC} gave lower and upper bounds on the Sombor index of graphs by using some graph parameters. Moreover, they obtained several relations on Sombor index with the first and second Zagreb indices of graphs. R\'{e}ti et al. \cite{RDA} characterized graphs with the maximum Sombor index in the classes of all connected unicyclic, bicyclic, tricyclic, tetracyclic, and pentacyclic graphs of a fixed order.  Red\v{z}epovi\'{c} \cite{R} showed that the Sombor index has good predictive potential. For other related results, one may refer to \cite{G5, K, KG, WMLF} and the references therein.

The aim of this paper is to study the Sombor index from an algebraic viewpoint, which is a natural idea in mathematical chemistry. The Sombor matrix $\mathcal{S}(G)$ of a graph $G$ is defined so that its $(i,j)$-entry is equal to $\sqrt{d_i^2+d_j^2}$ if the vertices $v_i$ and $v_j$ are adjacent, and zero otherwise, where $d_i$ denotes the degree of vertex $v_i$ in $G$. The eigenvalues of $\mathcal{S}(G)$ are denoted by $\rho_1(G)\geq \rho_2(G)\geq \cdots \geq \rho_n(G)$, where $\rho_1(G)$ is called the Sombor spectral radius of $G$. The Sombor energy and Estrada index of the Sombor matrix of graphs are defined as $\mathcal{E}(G)=\sum_{i=1}^{n}|\rho_i(G)|$ and $EE(G)=\sum_{i=1}^{n}e^{\rho_i(G)}$, respectively. We obtain lower and upper bounds on the  Sombor spectral radius,  Sombor energy and Sombor Estrada index of graphs, and characterize the respective extremal graphs.

\section{\large  Preliminaries}

The diameter of a graph $G$, denoted by $diam(G)$, is the maximum distance between any pair of vertices of $G$. Let $K_{s, \, t}$ and $K_n$ denote the complete bipartite graph with $s+t$ vertices and the complete graph with $n$ vertices, respectively. The first Zagreb index \cite{GT} $Z_1$ and forgotten topological index \cite{FG} $F$ of $G$ are defined as
$$Z_1=Z_1(G)=\sum\limits_{i=1}^{n}d_i^2=\sum\limits_{v_iv_j\in E(G)}(d_i+d_j), \quad F=F(G)=\sum\limits_{i=1}^{n}d_i^3=\sum\limits_{v_iv_j\in E(G)}(d_i^2+d_j^2).$$

\begin{lemma}{\bf (\cite{CS, Z})}\label{le2,1} 
Let $G$ be a graph with $n$ vertices. Then
$$\sqrt{\frac{Z_1}{n}}\leq \lambda_1 \leq \Delta.$$
The equality in the left hand side holds if and only if $G$ is regular or semiregular. If $G$ is a connected graph, then the equality in the right hand side holds if and only if $G$ is regular.
\end{lemma}

\begin{lemma}{\bf (\cite{CS, H1})}\label{le2,2} 
Let $G$ be a connected graph of order $n$ with $m$ edges. Then
$$\frac{2m}{n}\leq \lambda_1 \leq \sqrt{2m-n+1}.$$
The equality in the left hand side holds if and only if $G$ is a regular graph, and the equality in the right hand side holds if and only if $G\cong K_{1,\,n-1}$ or $G\cong K_n$.
\end{lemma}

\begin{lemma}{\bf (\cite{H})}\label{le2,3} 
Let $A$ be a symmetric matrix of order $n$ with eigenvalues $\xi_1\geq \xi_2 \geq \cdots  \geq \xi_n$, and let $B$ be
its principal submatrix with eigenvalues $\eta_1\geq \eta_2 \geq \cdots  \geq \eta_k$ and $n>k$. Then $\xi_i \geq \eta_i \geq \xi_{n-k+i}$
for $i=1, 2, \ldots, k$.
\end{lemma}

\begin{lemma}{\bf (\cite{K1})}\label{le2,4} 
Let $a_1\geq a_2\geq \cdots \geq a_n\geq 0$ be a sequence of non-negative real numbers. Then
$$\sum\limits_{i=1}^{n}a_i+n(n-1)\left(\prod_{i=1}^{n}a_i\right)^{1/n}\leq \left(\sum\limits_{i=1}^{n}\sqrt{a_i}\right)^2\leq (n-1)\sum\limits_{i=1}^{n}a_i+n\left(\prod_{i=1}^{n}a_i\right)^{1/n}.$$
\end{lemma}

\begin{lemma}{\bf (\cite{GG})}\label{le2,5} 
Let $M$ be an $n\times n$ non-negative symmetric matrix such that its underlying graph is connected. Let $\lambda_1(M), \lambda_2(M), \ldots, \lambda_k(M)$ be all the eigenvalues of $M$ with absolute value equal to $\lambda_1(M)$. Then $k>1$ if and only if all closed walks in $G$ have length divisible by $k$.
\end{lemma}

\begin{lemma}{\bf (\cite{GP, LSG})}\label{le2,6} 
Let $T_n$ be a tree with $n\geq 3$ vertices. Then
$$2\sqrt{n-1}=E_A(K_{1,\,n-1})\leq E_A(T_n)\leq E_A(P_n)=\begin{dcases}
2\csc\frac{\pi}{2(n+1)}-2, & \text{if}\,\, n \,\,\text{is} \,\,\text{even},\\
2\cot\frac{\pi}{2(n+1)}-2, & \text{if}\,\, n\,\,\text{is} \,\,\text{odd}.\\
\end{dcases}$$
\end{lemma}

\begin{lemma}\label{le2,7} 
If $G$ is a connected graph with $k\geq 2$ distinct Sobmor eigenvalues, then $diam(G)\leq k-1$.
\end{lemma}

\begin{proof} Let $\mathcal{S}$ be the Sombor matrix of $G$ and $\rho_1>\rho_2>\cdots>\rho_k$ be its $k$ distinct Sombor eigenvalues. Let $X$ be the unit (column) eigenvector corresponding the largest eigenvalues $\rho_1$. Then $X$ is a positive vector. From Theorem 2.1 in \cite{LS},  it follows that
$$\prod_{i=2}^{k}(\mathcal{S}-\rho_iI)=\mathcal{S}^{k-1}+c_1\mathcal{S}^{k-2}+c_{k-2}\mathcal{S}+c_{k-1}I=\prod_{i=2}^{k}(\rho_1-\rho_i)XX^T=M.$$
Observe that $(M)_{ij}>0$ for each $i, j =1, 2, \ldots, n$. Therefore, for $i\neq j$, there is a positive integer $l$ with $1\leq l \leq k-1$ such that $(\mathcal{S}^l)_{ij}>0$, which implies that there is a path of length $l$ between $v_i$ and $v_j$, that is, $diam(G)\leq k-1$. The proof is completed. $\Box$
\end{proof}

\begin{lemma}\label{le2,8} 
Let $G$ be a graph with $n$ vertices. Then $|\rho_1|=|\rho_2|=\cdots=|\rho_n|$ if and only if $G\cong \overline{K_n}$ or $G\cong \frac{n}{2}K_2$.
\end{lemma}

\begin{proof} First we assume that $|\rho_1|=|\rho_2|=\cdots=|\rho_n|$. Let $t$ be the number of isolated vertices in $G$. If $t\geq 1$, then $\rho_1=\rho_2=\cdots=\rho_n=0$ and hence $G\cong \overline{K_n}$. Otherwise, $t=0$. If $\Delta=1$, then $d_1=d_2=\cdots=d_n=1$ and hence $G\cong \frac{n}{2}K_2$. Otherwise, $\Delta\geq 2$. Then $G$ contains a connected component $H$ with at least $3$ vertices. If $H$ is a complete graph with $p\geq 3$ vertices, then $|\rho_1(H)|=\sqrt{2}(p-1)^2>\sqrt{2}(p-1)=|\rho_2(H)|$, a contradiction. Otherwise, $H$ is not a complete graph. By Lemma \ref{le2,3}, $\rho_2(H)\geq 0$. By the Perron-Frobenius theorem, $\rho_1(H)>\rho_2(H)$, a contradiction.
Conversely, one can easily check that $|\rho_1|=|\rho_2|=\cdots=|\rho_n|$ holds for $\overline{K_n}$ and $\frac{n}{2}K_2$.
This completes the proof. $\Box$
\end{proof}

\begin{lemma}\label{le2,9} 
Let $G$ be a graph with Sombor eigenvalues $\rho_1, \rho_2, \ldots, \rho_n$. Then

{\normalfont (i)} $\sum\limits_{i=1}^{n}\rho_i=0,\quad \sum\limits_{i=1}^{n}\rho_i^2=-2\sum\limits_{1\leq i<j\leq n}\rho_i\rho_j, \quad \sum\limits_{i=1}^{n}\rho_i^2=tr(\mathcal{S}^2)=2\sum\limits_{v_iv_j\in E(G)}(d_i^2+d_j^2)=2F$.

{\normalfont (ii)} $\sum\limits_{i=1}^{n}\rho_i^3=tr(\mathcal{S}^3)=2\sum\limits_{i\sim j}\sqrt{d_i^2+d_j^2}\sum\limits_{k\sim i,\, k\sim j}\sqrt{(d_i^2+d_k^2)(d_j^2+d_k^2)}$.

{\normalfont (iii)} $\sum\limits_{i=1}^{n}\rho_i^4=tr(\mathcal{S}^4)  =  \sum\limits_{i=1}^{n}\left(\sum\limits_{i\sim k}(d_i^2+d_k^2)\right)^2+\sum\limits_{i\neq j}\left(\sum\limits_{k\sim i,\, k\sim j}\sqrt{(d_i^2+d_k^2)(d_j^2+d_k^2)}\right)^2$.

\end{lemma}

\begin{proof} {\normalfont (i)} It follows from definition.

{\normalfont (ii)} Suppose $i\neq j$. Then
$$(\mathcal{S}^2)_{ij}=\sum\limits_{k=1}^{n}(\mathcal{S})_{ik}(\mathcal{S})_{kj}=\sum\limits_{k\sim i, \, k\sim j}(\mathcal{S})_{ik}(\mathcal{S})_{kj}=\sum\limits_{k\sim i,\, k\sim j}\sqrt{(d_i^2+d_k^2)(d_j^2+d_k^2)}.$$
For the matrix $\mathcal{S}$, we have
\begin{eqnarray*}
(\mathcal{S}^3)_{ii} & = & \sum\limits_{j=1}^{n}(\mathcal{S})_{ij}(\mathcal{S}^2)_{jk}=\sum\limits_{i\sim j}(\sqrt{d_i^2+d_j^2})(\mathcal{S}^2)_{jk}\\
& = & \sum\limits_{i\sim j}\sqrt{d_i^2+d_j^2}\sum\limits_{k\sim i,\, k\sim j}\sqrt{(d_i^2+d_k^2)(d_j^2+d_k^2)}.
\end{eqnarray*}
Thus
\begin{eqnarray*}
tr(\mathcal{S}^3) & = & \sum\limits_{i=1}^{n}\sum\limits_{i\sim j}\sqrt{d_i^2+d_j^2}\sum\limits_{k\sim i,\, k\sim j}\sqrt{(d_i^2+d_k^2)(d_j^2+d_k^2)}\\
& = & 2\sum\limits_{i\sim j}\sqrt{d_i^2+d_j^2}\sum\limits_{k\sim i,\, k\sim j}\sqrt{(d_i^2+d_k^2)(d_j^2+d_k^2)}.
\end{eqnarray*}

{\normalfont (iii)} The trace of $\mathcal{S}$ is
\begin{eqnarray*}
tr(\mathcal{S}^4) & = & \sum\limits_{i, j=1}^{n}(\mathcal{S}^2)_{ij}^2=\sum\limits_{i=j}(\mathcal{S}^2)_{ij}^2+\sum\limits_{i\neq j}(\mathcal{S}^2)_{ij}^2\\
& = & \sum\limits_{i=1}^{n}\left(\sum\limits_{i\sim k}(d_i^2+d_k^2)\right)^2+\sum\limits_{i\neq j}\left(\sum\limits_{k\sim i,\, k\sim j}\sqrt{(d_i^2+d_k^2)(d_j^2+d_k^2)}\right)^2.
\end{eqnarray*}
This completes the proof. $\Box$
\end{proof}

\section{\large  On spectral radius of the Sombor matrix}

\begin{theorem}\label{th3,1} 
Let $G$ be a graph with $n$ vertices, the maximum degree $\Delta$ and minimum degree $\delta$. Then
$$\sqrt{2}\delta\lambda_1\leq \rho_1 \leq \sqrt{2}\Delta\lambda_1 \eqno{(3.1)}$$
with equality if and only if $G$ is a regular graph.
\end{theorem}

\begin{proof} Firstly, we prove the left hand side of (3.1). Let $X=(x_1, x_2, \ldots, x_n)$ be a unit eigenvector of $G$ corresponding to $\lambda_1$. By the Rayleigh-Ritz Theorem, we have
$$\rho_1 \geq X^T\mathcal{S}X\geq 2\sum\limits_{v_iv_j\in E(G)}\sqrt{d_i^2+d_j^2}x_ix_j\geq 2\sqrt{2}\delta\sum\limits_{v_iv_j\in E(G)}x_ix_j=\sqrt{2}\delta\lambda_1.\eqno{(3.2)}$$

Now suppose that equality in the left hand side of (3.1) holds. Then inequality in (3.2) must be equality if and only if $d_i=d_j$ for each edge $v_iv_j\in E(G)$, that is, $G$ is regular. Clearly, if $G$ is a regular graph, the equality in the left hand side of (3.1) holds.

Secondly, we prove the right hand side of (3.1). Let $Y=(y_1, y_2, \ldots, y_n)$ be a unit eigenvector of $G$ corresponding to $\rho_1$. By the Rayleigh-Ritz Theorem, we have
$$\lambda_1 \geq Y^TAY\geq 2\sum\limits_{v_iv_j\in E(G)}y_iy_j.$$
Similarly, we have
$$\rho_1= Y^T\mathcal{S}Y= 2\sum\limits_{v_iv_j\in E(G)}\sqrt{d_i^2+d_j^2}y_iy_j\leq 2\sqrt{2}\Delta\sum\limits_{v_iv_j\in E(G)}y_iy_j\leq \sqrt{2}\Delta\lambda_1.\eqno{(3.3)}$$

Now suppose that equality in the right hand side of (3.1) holds. Then inequality in (3.3) must be equality if and only if $d_i=d_j$ for any $v_iv_j\in E(G)$, that is $G$ is a regular graph. Conversely, it is easy to verify that the equality in the right hand side of (3.1) holds when $G$ is a regular graph.

Combining the above arguments, we have the proof. $\Box$

\end{proof}

By Lemmas \ref{le2,1}, \ref{le2,2} and Theorem \ref{th3,1}, we have the following corollaries.

\begin{corollary}\label{cor3,1} 
Let $G$ be a connected graph with $n$ vertices, $m$ edges, the maximum degree $\Delta$ and minimum degree $\delta$. Then
$$\delta\sqrt{\frac{2Z_1}{n}}\leq \rho_1 \leq \sqrt{2}\Delta^2.$$
The equality holds if and only if $G$ is a regular graph.
\end{corollary}

\begin{corollary}\label{cor3,2} 
Let $G$ be a connected graph with $n$ vertices, $m$ edges, the maximum degree $\Delta$ and minimum degree $\delta$. Then
$$\frac{2\sqrt{2}m\delta}{n}\leq \rho_1 \leq \Delta \sqrt{4m-2n+2}.$$
The equality in the left hand side holds if and only if $G$ is a regular graph, and the equality in the right hand side holds if and only if $G\cong K_n$.
\end{corollary}

\begin{theorem}\label{th3,2} 
Let $G$ be a graph with $n$ vertices. Then
$$\sqrt{\frac{2F}{n}}\leq \rho_1 \leq \sqrt{\frac{2(n-1)F}{n}}.\eqno{(3.4)}$$
The equality in the left hand side holds if and only if $G\cong \overline{K_n}$ or $G\cong \frac{n}{2}K_2$. If $G$ is a connected graph, then the equality in the right hand side holds if and only if $G\cong K_n$.
\end{theorem}

\begin{proof} By Lemma \ref{le2,9}, we have $n\rho_1^2\geq 2F$. The equality holds if and only if $|\rho_1|=|\rho_2|=\cdots=|\rho_n|$, by Lemma \ref{le2,8}, we have the result. By the Cauchy-Schwarz inequality, we have
$$\rho_1^2=2F-\sum\limits_{i=2}^{n}\rho_i^2\leq 2F-\frac{1}{n-1}\left(\sum\limits_{i=2}^{n}\rho_i\right)^2=2F-\frac{1}{n-1}\rho_1^2.$$
Thus $\rho_1 \leq \sqrt{\frac{2(n-1)F}{n}}$. The equality holds if and only if $\rho_2=\rho_3=\cdots=\rho_n$, by Lemma \ref{le2,7}, we have $diam(G)=1$, that is $G\cong K_n$. Conversely, if $G\cong K_{n}$, then $\mathcal{S}(G)=\sqrt{2}(n-1)A(G)$. Thus $\rho_1(G)=\sqrt{2}(n-1)^2$, $\rho_2(G)=\cdots=\rho_{n}(G)=-\sqrt{2}(n-1)$. It is easy to check that equality holds in the right hand side in (3.4). This completes the proof. $\Box$
\end{proof}

\begin{theorem}\label{th3,3} 
Let $uv$ be an edge of a connected graph $G$ with $d_u\geq 2$ and $d_v\geq 2$. Let $X$ be a unit eigenvector of $G$ corresponding to $\rho_1(G)$, and let $N_G(u)\cap N_G(v)=\Phi$ and $G^{*}=G-\{vw:w\in N_G(v)\backslash \{u\}\}+\{uw:w\in N_G(v)\backslash \{u\}\}$. If $x_u\geq x_v$, then $\rho_1(G^{*})>\rho_1(G)$.
\end{theorem}

\begin{proof} Since $X$ is a unit eigenvector of $\rho_1(G)$, $X$ is positive if $G$ is connected. From the hypothesis, we have
\begin{eqnarray*}
\rho_1(G^{*})-\rho_1(G) & \geq & X^T\mathcal{S}(G^{*})X-X^T\mathcal{S}(G)X\\
\end{eqnarray*}
\begin{eqnarray*}
& = & 2\sum\limits_{p\in N_G(u)\backslash \{v\}}(\sqrt{d_{p}^2+(d_u+d_v-1)^2}-\sqrt{d_{p}^2+d_u^2})x_ux_{p}+\\
& & 2\sum\limits_{w\in N_G(v)\backslash \{u\}}(\sqrt{d_{w}^2+(d_u+d_v-1)^2}x_u-\sqrt{d_{w}^2+d_v^2}x_v)x_{w}+\\
& & 2(\sqrt{1^2+(d_u+d_v-1)^2}-\sqrt{d_{u}^2+d_v^2})x_{u}x_{v}\\
& \geq & 2\sum\limits_{p\in N_G(u)\backslash \{v\}}(\sqrt{d_{p}^2+(d_u+d_v-1)^2}-\sqrt{d_{p}^2+d_u^2})x_ux_{p}+\\
& & 2\sum\limits_{w\in N_G(v)\backslash \{u\}}(\sqrt{d_{w}^2+(d_u+d_v-1)^2}-\sqrt{d_{w}^2+d_v^2})x_vx_{w}+\\
& & 2(\sqrt{1^2+(d_u+d_v-1)^2}-\sqrt{d_{u}^2+d_v^2})x_{u}x_{v}\\
& > & 0.
\end{eqnarray*}
The proof is completed. $\Box$
\end{proof}

\begin{corollary}\label{cor3,3} 
Let $T_n$ be a tree with $n$ vertices. Then
$$\rho_1(T_n)\leq \sqrt{(n-1)(n^2-2n+2)}$$
with equality if and only if $G=K_{1,\, n-1}$.
\end{corollary}

Li and Wang \cite{LW} showed that the path $P_n$ is uniquely the tree on $n\geq 9$ vertices with the smallest Sombor spectral radius among all trees. Thus we have

\begin{theorem}{\bf (\cite{LW})}\label{th3,4} 
Let $T_n$ be a tree with $n\geq 9$ vertices. Then
$$\rho_1(P_n)\leq \rho_1(T_n)\leq \rho_1(K_{1,\, n-1}).$$
The equality in the left hand side holds if and only if $T_n\cong P_n$, and the equality in the right hand side holds if and only if $T_n\cong K_{1,\,n-1}$.
\end{theorem}

\begin{problem}
For a given class of graphs, characterize the graphs with the maximum or minimum Sombor spectral radius.
\end{problem}

\section{\large  On the Sombor energy}

\begin{theorem}\label{th4,1} 
Let $G$ be a graph with $n$ vertices. Then
$$\sqrt{2F+n(n-1)\left(|\det\mathcal{S}(G)|\right)^{2/n}}\leq \mathcal{E}(G)\leq \sqrt{2(n-1)F+n\left(|\det\mathcal{S}(G)|\right)^{2/n}}.$$
\end{theorem}

\begin{proof}
Replacing $a_i=\rho_i^2(G)$ in Lemma \ref{le2,4}, we have the proof. $\Box$
\end{proof}

\begin{theorem}\label{th4,2} 
Let $G$ be a graph with $n\geq 3$ vertices. Then
$$2\sqrt{F}\leq \mathcal{E}(G)\leq \sqrt{2nF}.\eqno{(4.1)}$$
The equality in the right hand side holds if and only if $G\cong \overline{K_n}$ or $G\cong \frac{n}{2}K_2$. If $G$ is a connected graph, then the equality in the left hand side holds if and only if $G\cong K_{s,\,t}$, $s+t=n$.
\end{theorem}

\begin{proof} By the Cauchy-Schwarz inequality, we have
$$\mathcal{E}(G)=\sum\limits_{i=1}^{n}|\rho_i|\leq\sqrt{n\sum\limits_{i=1}^{n}\rho_i^2}=\sqrt{2nF}.$$
with equality if and only if $|\rho_1|=|\rho_2|=\cdots=|\rho_n|$. By Lemma \ref{le2,8}, we have $G\cong \overline{K_n}$ or $G\cong \frac{n}{2}K_2$.
Since $\sum\limits_{i=1}^{n}\rho_i^2=-2\sum\limits_{1\leq i<j\leq n}\rho_i\rho_j$, we have
$$\mathcal{E}^2(G)=\sum\limits_{i=1}^{n}\rho_i^2+2\sum\limits_{1\leq i<j\leq n}|\rho_i||\rho_j|\geq \sum\limits_{i=1}^{n}\rho_i^2+2\left\lvert \sum\limits_{1\leq i<j\leq n}\rho_i\rho_j \right\rvert=2\sum\limits_{i=1}^{n}\rho_i^2=4F$$
with equality if and only if $\rho_1=-\rho_n$, $\rho_2=\cdots=\rho_{n-1}=0$. By Lemma \ref{le2,5}, we may conclude that all closed walks in $G$ have even length, which implies that $G$ is bipartite. By Lemma \ref{le2,7}, we have $diam=2$.
Thus $G$ is a complete bipartite graph $G\cong K_{s,\,t}$, $s+t=n$. Conversely, if $G\cong K_{s,\, t}$, then $\mathcal{S}(G)=\sqrt{s^2+t^2}A(G)$. Thus $\rho_1(G)=-\rho_n(G)=\sqrt{s^3t+st^3}=\sqrt{F}$, $\rho_2(G)=\cdots=\rho_{n-1}(G)=0$. It is easy to check that equality holds in the left hand side in (4.1).
This completes the proof. $\Box$
\end{proof}

\begin{corollary}\label{cor4,1} 
Let $G$ be a complete bipartite graph with $n$ vertices. Then
$$2\sqrt{(n-1)(n^2-2n+2)}\leq \mathcal{E}(G)\leq 2\sqrt{\left\lceil\frac{n}{2}\right\rceil^3\left\lfloor\frac{n}{2}\right\rfloor
+\left\lfloor\frac{n}{2}\right\rfloor\left\lceil\frac{n}{2}\right\rceil^3}.$$
The equality in the left hand side holds if and only if $G\cong K_{1,\, n-1}$, and the equality in the right hand side holds if and only if $G\cong K_{\lceil\frac{n}{2}\rceil,\, \lfloor\frac{n}{2}\rfloor}$.
\end{corollary}

\begin{proof} By the proof of Theorem \ref{th4,2}, we have
$$\mathcal{E}(K_{s,\,t})= 2\sqrt{s^3t+st^3}=2\sqrt{(n-t)^3t+(n-t)t^3}, \quad 1\leq t\leq \left\lfloor\frac{n}{2}\right\rfloor.$$
Let $f(x)=2\sqrt{(n-x)^3x+(n-x)x^3}$. By derivative, we know that $f(x)$ is a strictly increasing function in the interval $[1, \lfloor\frac{n}{2}\rfloor]$. Thus
$$\mathcal{E}(K_{1,\, n-1})=f(1)\leq \mathcal{E}(K_{s,\,t})= f(t)\leq f(\lfloor\frac{n}{2}\rfloor)=\mathcal{E}(K_{\lceil\frac{n}{2}\rceil,\, \lfloor\frac{n}{2}\rfloor}).$$
This completes the proof. $\Box$
\end{proof}

\begin{theorem}\label{th4,3} 
Let $G$ be a graph with $n$ vertices. Then
$$2\rho_1\leq \mathcal{E}(G)\leq \rho_1+\sqrt{(n-1)(2F-\rho_1^2)}.$$
\end{theorem}

\begin{proof} Since $\sum\limits_{i=1}^{n}\rho_i=0$, we have
$\mathcal{E}(G)=\rho_1+\sum\limits_{i=2}^{n}|\rho_i|\geq \rho_1+|\sum\limits_{i=2}^{n}\rho_i|=2\rho_1$.
Since $\mathcal{E}(G)=\rho_1+\sum\limits_{i=2}^{n}|\rho_i|$, by the Cauchy-Schwartz inequality, we have
$$\mathcal{E}(G)\leq \rho_1+\sqrt{(n-1)(2F-\rho_1^2)}.$$
This completes the proof. $\Box$
\end{proof}

\begin{theorem}\label{th4,4} 
Let $G$ be a non-trivial graph.Then
$$\mathcal{E}(G)\geq \sqrt{\frac{tr(\mathcal{S}^2)^3}{tr(\mathcal{S}^4)}}.$$
\end{theorem}

\begin{proof} Let $a_i=|\rho_i|^{\frac{2}{3}}$, $b_i=|\rho_i|^{\frac{4}{3}}$, $p=\frac{3}{2}$ and $q=3$ in the H\"{o}lder inequality
$$\sum\limits_{i=1}^{n}a_ib_i\leq \left(\sum\limits_{i=1}^{n}a_i^{p}\right)^{\frac{1}{p}} \left(\sum\limits_{i=1}^{n}b_i^{q}\right)^{\frac{1}{q}}.$$
Then
$$\sum\limits_{i=1}^{n}|\rho_i|^2=\sum\limits_{i=1}^{n}|\rho_i|^{\frac23}\left(|\rho_i|^{4}\right)^{\frac13}\leq \left(\sum\limits_{i=1}^{n}|\rho_i|\right)^{\frac{2}{3}} \left(\sum\limits_{i=1}^{n}|\rho_i|^{4}\right)^{\frac{1}{3}},$$
that is,
$$\mathcal{E}(G)\geq\left(\frac{\sum\limits_{i=1}^{n}|\rho_i|^2}{\left(\sum\limits_{i=1}^{n}|\rho_i|^{4}\right)^{\frac{1}{3}}}\right)^{\frac{3}{2}}
=\sqrt{\frac{tr(\mathcal{S}^2)^3}{tr(\mathcal{S}^4)}}.$$
The proof is completed. $\Box$
\end{proof}

For a graph $G$, we use $M_k(G)$ to denote the set of all $k$-matchings of $G$. If $e=v_iv_j\in E(G)$, then we denote
$$SO_G(e)=SO_G(v_iv_j)=\left(\sqrt{d_i^2+d_j^2} \right)^2=d_i^2+d_j^2,$$
and we say that $SO_G(e)$ is the $SO$-value of the edge $e$. If $\beta=\{e_1, e_2, \ldots, e_k\}\in M_k(G)$, we call that $\prod_{i=1}^{k}SO_G(e_i)$ is the
$SO$-value of matching $\beta$, and write $SO_G(\beta)=\prod_{i=1}^{k}SO_G(e_i)$.

If $G$ is a bipartite graph with $n$ vertices, then the characteristic polynomial of $G$ can be written as (see \cite{GP, LSG})
$$\phi_A(G,x)=|xI-A(G)|=\sum\limits_{k=0}^{\lfloor\frac{n}{2}\rfloor}(-1)^km(G,k)x^{n-2k},$$
where $m(G,0)=1$, and $m(G, k)$ equals the number of $k$-matchings of $G$ for $1\leq k \leq \lfloor\frac{n}{2}\rfloor$. The energy of $G$ can be expressed
as the Coulson integral formula (see \cite{GP, LSG})
$$E_A(G)=\frac{1}{\pi}\int_{-\infty}^{+\infty}\frac{1}{x^2}\ln\left(1+\sum\limits_{k=1}^{\lfloor\frac{n}{2}\rfloor}m(G,k)x^{2k}\right)dx.$$
Then $E_A(G)$ is a strictly monotonously increasing function of $m(G, k)$.

For a bipartite graph $G$ with $n$ vertices, the adjacency matrix $A(G)$ and Sombor matrix $\mathcal{S}(G)$ are nonnegative real symmetric matrices with zero diagonal, and$A(G)$ and $\mathcal{S}(G)$ have the same zero-nonzero pattern, that is, for any $1\leq i, j \leq n$,
$(i, j)$-entry of $A(G)$ is nonzero (or zero) if and only if $(i, j)$-entry of $\mathcal{S}(G)$ is nonzero (or zero). Thus the Sombor characteristic polynomial of $G$ can be written as
$$\phi_{SO}(G,x)=|xI-\mathcal{S}(G)|=\sum\limits_{k=0}^{\lfloor\frac{n}{2}\rfloor}(-1)^kb(\mathcal{S}(G),k)x^{n-2k},$$
where $b(\mathcal{S}(G),0)=1$, and $b(\mathcal{S}(G), k)$ equals the sum of $SO$-values of all $k$-matchings of $G$ for $1\leq k \leq \lfloor\frac{n}{2}\rfloor$. Similarly, the Coulson integral formula for Sombor energy of a bipartite graph $G$ can be expressed as follows
$$\mathcal{E}(G)=\frac{1}{\pi}\int_{-\infty}^{+\infty}\frac{1}{x^2}\ln\left(1+\sum\limits_{k=1}^{\lfloor\frac{n}{2}\rfloor}
b(\mathcal{S}(G),k)x^{2k}\right)dx.\eqno{(4.2)}$$
It is easy to see that $\mathcal{E}(G)$ is a strictly monotonously increasing function of $b(\mathcal{S}(G),k)$. So the following two results are direct.

\begin{proposition}\label{pro4,1} 
Let $G_1$ and $G_2$ be two bipartite graphs with $n$ vertices, and their Sombor characteristic polynomials be
$$\phi_{SO}(G_1,x)=\sum\limits_{k=0}^{\lfloor\frac{n}{2}\rfloor}(-1)^kb(\mathcal{S}(G_1),k)x^{n-2k},\quad \phi_{SO}(G_2,x)=\sum\limits_{k=0}^{\lfloor\frac{n}{2}\rfloor}(-1)^kb(\mathcal{S}(G_2),k)x^{n-2k},$$
respectively. If $b(\mathcal{S}(G_1),k)\geq b(\mathcal{S}(G_2),k)$ for all $k\geq 0$, and there is a positive integer $k$ such that $b(\mathcal{S}(G_1),k)> b(\mathcal{S}(G_2),k)$, then $\mathcal{S}(G_1)>\mathcal{S}(G_2)$.
\end{proposition}

\begin{proposition}\label{pro4,2} 
Let $G$ be a bipartite of with $n$ vertices, and both $B_1$ and $B_2$ be two $n\times n$ nonnegative real symmetric matrices having the same zero-nonzero pattern with $\mathcal{S}(G)$. If $B_1>\mathcal{S}(G)>B_2$, then
$$E_{B_1}>\mathcal{E}(G)>E_{B_2},$$
where $E_{B_1}$ is the sum of the absolute values of all eigenvalues of $B_1$.
\end{proposition}

\begin{theorem}\label{th4,5} 
Let $G$ be a bipartite graph with $n$ vertices. Then
$$\mathcal{E}(G-e)< \mathcal{E}(G).$$
\end{theorem}

\begin{proof}
Since $\mathcal{S}(G-e)<\mathcal{S}(G)$ , by Proposition \ref{pro4,2}, we have the proof. $\Box$
\end{proof}

The inverse sum indeg index, introduced by Vuki\v{c}evi\'{c} and Ga\v{s}perov \cite{VG}, is defined as
$$ISI(G)=\sum\limits_{v_iv_j\in E(G)}\frac{d_id_j}{d_i+d_j}.$$
Xu et al. \cite{XLYZ} obtained lower and upper bounds on the inverse sum indeg energy $E_{ISI}(G)$ of a graph $G$, and characterized the respective extremal graphs.

\begin{theorem}\label{th4,6} 
Let $G$ be a bipartite graph with $n$ vertices. Then
$$\mathcal{E}(G)\geq 2\sqrt{2}E_{ISI}(G)$$
with equality if and only if $G$ is regular.
\end{theorem}

\begin{proof}
Since $\sqrt{d_i^2+d_j^2}\geq \frac{2\sqrt{2}d_id_j}{d_i+d_j}$ with equality if and only if $G$ is regular, by Proposition \ref{pro4,2}, we have the proof. $\Box$
\end{proof}

\begin{theorem}\label{th4,7} 
Let $T_n$ be a tree with $n$ vertices and maximum degree $\Delta$. Then
$$2\sqrt{2n-2}<\mathcal{E}(T_n)<\begin{dcases}
2\sqrt{2}\Delta\csc\frac{\pi}{2(n+1)}-2\sqrt{2}\Delta, & \text{if}\,\, n \,\,\text{is} \,\,\text{even},\\
2\sqrt{2}\Delta\cot\frac{\pi}{2(n+1)}-2\sqrt{2}\Delta, & \text{if}\,\, n\,\,\text{is} \,\,\text{odd}.
\end{dcases}$$
\end{theorem}

\begin{proof}
Since $\sqrt{2}<\sqrt{d_i^2+d_j^2}\leq \sqrt{2}\Delta$ for a tree, by Proposition \ref{pro4,2}, we have $\sqrt{2}E_A(T_n)<\mathcal{E}(T_n)<\sqrt{2}\Delta E_A(T_n)$. By Lemma \ref{le2,6}, we have the proof. $\Box$
\end{proof}

\begin{problem}
For a given class of graphs, characterize the graphs with the maximum or minimum Sombor energy.
\end{problem}

\section{\large On the Sombor Estrada index}

\begin{theorem}\label{th5,1} 
Let $G$ be a graph with $n$ vertices and $m$ edges. Then
$$EE(G)\leq n-1+\frac{tr(\mathcal{S}^3)}{6}+\frac{tr(\mathcal{S}^4)}{24}+e^{\sqrt{2F}}-\sqrt{2F}
-\frac{1}{3}F\sqrt{2F}-\frac{1}{6}F^2.$$
Equality holds if and only if $G$ is an empty graph.
\end{theorem}

\begin{proof} By the definition of the Sombor Estrada index of graphs, we have
\begin{eqnarray*}
EE(G) & \leq & n+\sum\limits_{t=1}^{n}\rho_t(G)+\sum\limits_{t=1}^{n}\frac{\rho_t^2(G)}{2!}+\sum\limits_{t=1}^{n}\frac{\rho_t^3(G)}{3!}
+\sum\limits_{t=1}^{n}\frac{\rho_t^4(G)}{4!}+\sum\limits_{t=1}^{n}\sum\limits_{k\geq 5}^{\infty}\frac{|\rho_t(G)|^k}{k!}\\
& = & n+F+\frac{tr(\mathcal{S}^3)}{6}+\frac{tr(\mathcal{S}^4)}{24}+\sum\limits_{t=1}^{n}\sum\limits_{k\geq 5}^{\infty}\frac{|\rho_t(G)|^k}{k!}\\
& = & n+F+\frac{tr(\mathcal{S}^3)}{6}+\frac{tr(\mathcal{S}^4)}{24}+\sum\limits_{k\geq 5}^{\infty}\frac{1}{k!}\sum\limits_{t=1}^{n}|\rho_t(G)|^k\\
\end{eqnarray*}
\begin{eqnarray*}
& = & n+F+\frac{tr(\mathcal{S}^3)}{6}+\frac{tr(\mathcal{S}^4)}{24}+\sum\limits_{k\geq 5}^{\infty}\frac{1}{k!}\sum\limits_{t=1}^{n}(\rho_t^2(G))^{\frac{k}{2}}\\
& \leq & n+F+\frac{tr(\mathcal{S}^3)}{6}+\frac{tr(\mathcal{S}^4)}{24}+\sum\limits_{k\geq 5}^{\infty}\frac{1}{k!}(\sum\limits_{t=1}^{n}\rho_t^2(G))^{\frac{k}{2}}\\
& = & n+F+\frac{tr(\mathcal{S}^3)}{6}+\frac{tr(\mathcal{S}^4)}{24}+\sum\limits_{k=0}^{\infty}\frac{(2F)^{\frac{k}{2}}}{k!}-1-\sqrt{2F}-F
-\frac{1}{3}F\sqrt{2F}-\frac{1}{6}F^2\\
& = & n-1+\frac{tr(\mathcal{S}^3)}{6}+\frac{tr(\mathcal{S}^4)}{24}+e^{\sqrt{2F}}-\sqrt{2F}
-\frac{1}{3}F\sqrt{2F}-\frac{1}{6}F^2.
\end{eqnarray*}
This completes the proof. $\Box$
\end{proof}

\begin{theorem}\label{th5,2} 
Let $G$ be a graph with $n$ vertices. Then
$$EE(G)\geq e^{\rho_1(G)}+n_0+(p-1)e^{\frac{\frac{\mathcal{E}(G)}{2}-\rho_1(G)}{p-1}}
+qe^{-\frac{\mathcal{E}(G)}{2q}}\eqno{(5.1)}$$
with equality if and only if $\rho_2(G)=\cdots=\rho_p(G)$ and $\rho_{n-q+1}(G)=\cdots=\rho_n(G)$, where $p$, $n_0$ and $q$ are the number of positive, zero and negative Sombor eigenvalues of $G$, respectively.
\end{theorem}

\begin{proof}
Let $\rho_1(G)\geq \rho_2 (G)\geq \cdots \geq \rho_p(G)$ be the positive, and $\rho_{n-q+1}(G)\geq \rho_{n-q+2} (G)\geq \cdots \geq \rho_n(G)$ be the negative Sombor eigenvalues of $G$. By the arithmetic-geometric mean inequality, we have
$$\sum\limits_{i=2}^{p}e^{\rho_i(G)}\geq (p-1)e^{\frac{\rho_2(G)+\cdots+\rho_p(G)}{p-1}}
=(p-1)e^{\frac{\frac{\mathcal{E}(G)}{2}-\rho_1(G)}{p-1}}.\eqno{(5.2)}$$
Similarly,
$$\sum\limits_{i=n-q+1}^{n}e^{\rho_i(G)}\geq qe^{-\frac{\mathcal{E}(G)}{2q}}.\eqno{(5.3)}$$
For the zero eigenvalues, we have
$$\sum\limits_{p+1}^{n-q}e^{\rho_i(G)}=n_0.$$
Thus we have
$$EE(G)\geq e^{\rho_1(G)}+n_0+(p-1)e^{\frac{\frac{\mathcal{E}(G)}{2}-\rho_1(G)}{p-1}}
+qe^{-\frac{\mathcal{E}(G)}{2q}}.$$

The equality holds in $(5.1)$ if and only if equality holds in both $(5.2)$ and $(5.3)$ and these
happen if and only if $\rho_2(G)=\cdots=\rho_p(G)$ and $\rho_{n-q+1}(G)=\cdots=\rho_n(G)$. This completes the
proof. $\Box$
\end{proof}

\begin{theorem}\label{th5,3} 
Let $G$ be a bipartite graph with $n$ vertices. Then
$$EE(G)\geq n_0+2\cosh(\rho_1(G))+(r-2)\cosh\left(\frac{\mathcal{E}(G)-2\rho_1(G)}{r-2}\right)\eqno{(5.4)}$$
with equality if and only if $\rho_2(G)=\cdots=\rho_p(G)$, where $r$ is the rank of Sombor matrix.
\end{theorem}

\begin{proof} Since $G$ is bipartite, we have that its Sombor eigenvalues are symmetric with respect to zero, i.e. $\rho_{i}(G)=-\rho_{n-i+1}(G)$ for $i=1, 2, \ldots, \lfloor\frac{n}{2}\rfloor$.
By a similar argument as the proof of Theorem \ref{th5,2}, we have
\begin{eqnarray*}
EE(G) & = & n_0+e^{\rho_1(G)}+e^{-\rho_1(G)}+\sum\limits_{i=2}^{p}e^{\rho_i(G)}
+\sum\limits_{i=2}^{p}e^{-\rho_i(G)}\\
& \geq & n_0+e^{\rho_1(G)}+e^{-\rho_1(G)}\\
& & +(p-1)\left(e^{\frac{\frac{\mathcal{E}(G)}{2}-\rho_1(G)}{p-1}}
+e^{-\frac{\frac{\mathcal{E}(G)}{2}-\rho_1(G)}{p-1}}\right)\\
& = & n_0+2\cosh(\rho_1(G))+(r-2)\cosh\left(\frac{\mathcal{E}(G)-2\rho_1(G)}{r-2}\right).
\end{eqnarray*}
Note that $r=2p$. Equality holds in $(5.4)$ if and only if $\rho_2(G)=\cdots=\rho_p(G)$. The proof is completed. $\Box$
\end{proof}

\begin{theorem}\label{th5,4} 
Let $G$ be a connected bipartite graph with $n\geq 4$ vertices. Then
$$EE(G)\leq n-2+2\cosh\sqrt{F}\eqno{(5.5)}$$
with equality if and only if $G\cong K_{s,\, t}$, $s+t=n$.
\end{theorem}

\begin{proof} Since $G$ is bipartite, we have that its Sombor eigenvalues are symmetric with respect to zero. Let $p$ and $n_0$ be the number of positive and   zero Sombor eigenvalues of $G$, respectively. Then we have
\begin{eqnarray*}
EE(G) & = & n_0+\sum\limits_{i=1}^{p}\left(e^{\rho_i(G)}
+e^{-\rho_i(G)}\right)\\
& = & n_0+2p+2\sum\limits_{k=1}^{\infty}\frac{\sum\limits_{i=1}^{p}\rho_i^{2k}(G)}{(2k)!}\\
& \leq & n+2\sum\limits_{k=1}^{\infty}\frac{\left(\sum\limits_{i=1}^{p}\rho_i^{2}(G)\right)^k}{(2k)!}\\
\end{eqnarray*}
\begin{eqnarray*}
& = & n-2+2\sum\limits_{k=0}^{\infty}\frac{(\sqrt{F})^{2k}}{(2k)!}\\
& = & n-2+e^{\sqrt{F}}+e^{-\sqrt{F}}\\
& = & n-2+2\cosh\sqrt{F}.
\end{eqnarray*}
If equality holds in (5.5), then
$$\sum\limits_{i=1}^{p}\rho_i^{2k}(G)=\left(\sum\limits_{i=1}^{p}\rho_i^{2}(G)\right)^k$$
for $k\geq 1$. Since $G$ is a connected graph with $n\geq 4$ vertices, we have $\rho_1(G)>0$, that is $p\geq 1$.
For $k\geq 2$,
$$\sum\limits_{i=1}^{p}\rho_i^{2k}(G)=\left(\sum\limits_{i=1}^{p}\rho_i^{2}(G)\right)^k$$
implies that $p\leq 1$, as $\rho_i(G)$'s are positive eigenvalues.
Thus $p=1$, that is $\rho_1(G)=-\rho_n(G)=\sqrt{F}$, $\rho_2(G)=\cdots=\rho_{n-1}(G)=0$. By Lemma \ref{le2,7}, we have $diam(G)=2$. Thus $G$ is a complete bipartite graph $K_{s,\, t}$, $s+t=n$. Conversely, if $G\cong K_{s,\, t}$, then $\mathcal{S}(G)=\sqrt{s^2+t^2}A(G)$. Thus $\rho_1(G)=-\rho_n(G)=\sqrt{s^3t+st^3}=\sqrt{F}$, $\rho_2(G)=\cdots=\rho_{n-1}(G)=0$. It is easy to check that equality holds in (5.5). The proof is completed. $\Box$
\end{proof}

\begin{corollary}\label{cor5,1} 
Let $G$ be a complete bipartite graph with $n\geq 4$ vertices. Then
$$n-2+2\cosh\sqrt{(n-1)(n^2-2n+2)}\leq EE(G)\leq n-2+2\cosh\sqrt{\left\lceil\frac{n}{2}\right\rceil^3\left\lfloor\frac{n}{2}\right\rfloor
+\left\lfloor\frac{n}{2}\right\rfloor\left\lceil\frac{n}{2}\right\rceil^3}$$
The equality in the left hand side holds if and only if $G\cong K_{1,\, n-1}$, and the equality in the right hand side holds if and only if $G\cong K_{\lceil\frac{n}{2}\rceil,\, \lfloor\frac{n}{2}\rfloor}$.
\end{corollary}

\begin{proof} By the proof of Theorem \ref{th5,4}, we have
$$EE(K_{s,\,t})= n-2+2\cosh\sqrt{s^3t+st^3}=n-2+2\cosh\sqrt{(n-t)^3t+(n-t)t^3}, \, 1\leq t\leq \left\lfloor\frac{n}{2}\right\rfloor.$$
Let $f(x)=n-2+2\cosh\sqrt{(n-x)^3x+(n-x)x^3}$. By derivative, we know that $f(x)$ is a strictly increasing function in the interval $[1, \lfloor\frac{n}{2}\rfloor]$. Thus
$$EE(K_{1,\, n-1})=f(1)\leq EE(K_{s,\,t})= f(t)\leq f(\lfloor\frac{n}{2}\rfloor)=EE(K_{\lceil\frac{n}{2}\rceil,\, \lfloor\frac{n}{2}\rfloor}).$$
This completes the proof. $\Box$
\end{proof}

\small {

}

\end{document}